\documentclass[10pt,leqno,twoside,draft]{amsart}
\usepackage{amsmath,amssymb,amsthm,enumerate}
\usepackage{accents}  
\usepackage[usenames]{color}

\theoremstyle{plain}
\newtheorem{pretheo}{Theorem}[section]
\newtheorem{preassu}[pretheo]{Assumption}
\newtheorem{precoro}[pretheo]{Corollary}
\newtheorem{predefi}[pretheo]{Definition}
\newtheorem{preexam}[pretheo]{Example}
\newtheorem{prelemm}[pretheo]{Lemma}
\newtheorem{preprop}[pretheo]{Proposition}
\newtheorem{prerema}[pretheo]{Remark}
\newtheorem{preremas}[pretheo]{Remarks}

\newenvironment{theo}{\begin{pretheo}}{\end{pretheo}}

\newenvironment{coro}{\begin{precoro}}{\end{precoro}}

\newenvironment{lemm}{\begin{prelemm}}{\end{prelemm}}
\newenvironment{prop}{\begin{preprop}}{\end{preprop}}
\newenvironment{rema}{\begin{prerema}\rm}{\end{prerema}}
\newenvironment{remarks}{\begin{preremas}\rm}{\end{preremas}}

\DeclareMathOperator{\di}{div}

\newcommand{\jump}[1]{\ensuremath{[\![#1]\!]}}

\newcommand{\pa}{\partial}
\newcommand{\pr}{\prime}

\newcommand{\wt}[1]{\widetilde{#1}}

\makeatletter
\newcommand{\opnorm}{\@ifstar\@opnorms\@opnorm}
\newcommand{\@opnorms}[1]{%
  \left|\mkern-1.5mu\left|\mkern-1.5mu\left|
   #1
  \right|\mkern-1.5mu\right|\mkern-1.5mu\right|
}
\newcommand{\@opnorm}[2][]{%
  \mathopen{#1|\mkern-1.5mu#1|\mkern-1.5mu#1|}
  #2
  \mathclose{#1|\mkern-1.5mu#1|\mkern-1.5mu#1|}
}
\makeatother

\newcommand{\BBE}{\mathbb{E}}
\newcommand{\BBF}{\mathbb{F}}

\newcommand{\BBI}{\mathbb{I}}

\newcommand{\BBN}{\mathbb{N}}

\newcommand{\BBR}{\mathbb{R}}

\newcommand{\R}{\mathbb{R}}

\newcommand{\Be}{e}
\newcommand{\Bf}{f}
\newcommand{\Bg}{g}

\newcommand{\Bn}{n}

\newcommand{\Bu}{u}
\newcommand{\Bv}{v}

\newcommand{\Bz}{\mathbf{z}}

\newcommand{\BA}{A}

\newcommand{\BD}{D}
\newcommand{\BE}{E}
\newcommand{\BF}{F}
\newcommand{\BG}{G}

\newcommand{\BI}{I}

\newcommand{\BN}{N}

\newcommand{\BR}{\mathbf{R}}

\newcommand{\BT}{T}

\newcommand{\BV}{V}

\newcommand{\CA}{\mathcal{A}}
\newcommand{\CB}{\mathcal{B}}

\newcommand{\CE}{\mathcal{E}}
\newcommand{\CF}{\mathcal{F}}

\newcommand{\CH}{\mathcal{H}}

\newcommand{\CJ}{\mathcal{J}}

\newcommand{\CL}{\mathcal{L}}

\newcommand{\bdry}{\mathbb{R}_{0}^{N}}
\newcommand{\dws}{\dot{\mathbb{R}}^{N}}
\newcommand{\lhs}{\mathbb{R}_{-}^{N}}
\newcommand{\uhs}{\mathbb{R}_{+}^{N}}
\newcommand{\tws}{\mathbb{R}^{N-1}}
\newcommand{\ws}{\mathbb{R}^{N}}

\newcommand{\al}{\alpha}
\newcommand{\ga}{\gamma}
\newcommand{\de}{\delta}
\newcommand{\ep}{\varepsilon}
\newcommand{\te}{\theta}

\newcommand{\si}{\sigma}
\newcommand{\ph}{\varphi}

\newcommand{\Ga}{\Gamma}
\newcommand{\De}{\Delta}
\newcommand{\Te}{\Theta}
\newcommand{\La}{\Lambda}

\newcommand{\Om}{\Omega}

\newcommand{\Btau}{\tau}
\newcommand{\BCA}{\CA}
\newcommand{\BCB}{\CB}

\begin{document}

\title[Strong Solutions to two-phase free boundary problems]{Strong Solutions for two-phase free boundary problems  for a class of Non-Newtonian fluids}

\author[Matthias Hieber]{Matthias Hieber}
\address{Department of Mathematics, TU Darmstadt, Schlossgartenstr. 7, 64289 Darmstadt, Germany}
\email{hieber@mathematik.tu-darmstadt.de}

\author[Hirokazu Saito]{Hirokazu Saito}
\address{Department of Pure and Applied Mathematics, Graduate School of Fundamental Science and Engineering,
	Waseda University, Okubo 3-4-1, Shinjuku-ku, Tokyo 169-8555, Japan}
\email{s-hirokazu0704@toki.waseda.jp}



\thanks{This work was supported by the Japanese-German Graduate Externship JGGE. The second author is a JSPS Research Fellow.}

\date{}

\dedicatory{}

\begin{abstract}
Consider the two-phase free boundary problem subject to surface tension
and gravitational forces for a class of non-Newtonian fluids with  stress tensors $\BT_i$
of the form $\BT_i=-\pi\BI+\mu_i(|\BD(\Bv)|^2)\BD(\Bv)$  for $i=1,2$, respectively, and 
where the viscosity functions $\mu_i$ satisfy  $\mu_i(s)\in C^3([0,\infty))$ and $\mu_i(0)>0$ for $i=1,2$. 
It is shown that for given $T>0$ this problem admits a unique, strong solution on $(0,T)$ 
provided the initial data are sufficiently small in their natural norms.  
\end{abstract}

\maketitle
\renewcommand{\thefootnote}{\arabic{footnote})}
\numberwithin{equation}{section}

\section{Introduction and main result}\label{sec:intro}

The free boundary problem for two-phase flows for Newtonian fluids with or without surface tension is nowadays rather well understood. We refer in particular to the articles 
\cite{P-S1}, \cite{P-S2}, \cite{P-S3}, \cite{PSSS}, \cite{PSZ}, \cite{Abels}, \cite{AR}, \cite{AL} and \cite{SS07} describing the present state of research for the situation of 
sharp interfaces. 

In order to describe  the problem in more detail, let $N\geq 2$ and $\Gamma_0 \subset \ws$ be a surface which separates a region $\Omega_1(0)$ filled with 
a viscous, incompressible fluid from $\Omega_2(0)$, the complement of $\overline{\Omega_1(0)}$
in $\ws$. The region $\Omega_2(0)$ is also occupied with a second incompressible, viscous fluid
and it is assumed that the two fluids are immiscible. Denoting  by $\Gamma(t)$ the position
of $\Gamma_0$ at time $t$, $\Gamma(t)$ is then the interface separating the two fluids 
occupying the regions $\Omega_1(t)$ and $\Omega_2(t)$. 

An incompressible fluid is subject to the set of equations
\begin{eqnarray*}
\rho(\pa_t\Bu + \Bu\cdot\nabla\Bu) &=& \di\BT, \\
                \di\Bu &=& 0,
\end{eqnarray*}
where $\rho$ denotes the density of the fluid and the stress tensor $\BT$ can be decomposed as $\BT = \Btau - p\BI$, where $p$ denotes the pressure $p$ and 
$\Btau$ the tangential part of the stress tensor of the fluid.
For a Newtonian fluid, $\Btau$ is given by $\Btau = 2 \mu \BD(\Bu)$,
where $\BD(\Bu)= [\nabla\Bu + (\nabla\Bu)^T]/2$ denotes  the deformation tensor
and $\mu$ the viscosity coefficient of the fluid.  

In this article  we consider  a class of non-Newtonian fluids, where $\Btau$ as above is replaced by 
$$ 
\Btau = 2 \mu(|\BD(\Bu)|^2)\BD(\Bu)
$$ 
for some  function $\mu$ satisfying 
\begin{equation}\label{viscosity}
\mu \in C^3([0,\infty))\quad \text{ and }\quad \mu(0)>0.
\end{equation}
In the special case of power law fluids, one has
$$
\mu(|\BD(\Bu)|^2) = \nu + \beta |\BD(\Bu)|^{d-2} 
$$
for some $d\geq 1$ and constants $\nu, \beta \geq 0$. If $d<2$, the fluid is then called a shear thinning fluid, if $d>2$ it is called a shear thickening fluid.
Fluids of this type are special cases of so called {\it Stokesian fluids}, which were investigated analytically for fixed domains by Amann in \cite{Ama94} and \cite{Ama96}.   


The motion of the two immiscible, incompressible and viscous fluids is then governed by the set of equations  
\begin{equation}\label{NS}
	\left\{\begin{aligned}
		\rho(\pa_{t}\Bv+\Bv\cdot\nabla\Bv)
			&=\di\BT-\rho\ga_a\Be_N, && \text{in $\Om(t)$,} \\
		\di\Bv
			&=0 && \text{in $\Om(t)$,} \\
		-\jump{\BT\Bn_\Ga}
			&=\si H_\Ga\Bn_\Ga && \text{on $\Ga(t)$,} \\
		\jump{\Bv}
			&=0 && \text{on $\Ga(t)$,} \\
		\BV_\Ga
			&=\Bv\cdot\Bn_\Ga && \text{on $\Ga(t)$,} \\
		\Bv|_{t=0}
			&=\Bv_0 && \text{in $\Om_0$,} \\
		\Ga|_{t=0}
			&=\Ga_0,
	\end{aligned}\right.
\end{equation}
where $\Om(t)=\Om_1(t)\cup\Om_2(t)$ and $\Be_N=(0,\dots,0,1)^T$.
Here, the normal field  on $\Gamma(t)$, pointing from $\Omega_1(t)$ into $\Omega_2(t)$, is denoted by $\Bn_{\Ga}(t,\cdot)$.
Moreover, $\BV_\Ga(t,\cdot)$ and $H_\Ga(t,\cdot)$ denote the normal velocity and mean curvature of $\Gamma(t)$, respectively. 
Furthermore, $\ga_a$ denotes the gravitational acceleration and $\si$ the coefficient of the surface tension.  

In this article we suppose that the stress tensor $\BT$ is given by the generalized Newtonian type described above,
that is, for given scalar functions $\mu_1,\mu_2:[0,\infty) \to \R$, the stress tensor 
$\BT$ is given by 
\begin{equation*}
	\BT
		=\chi_{\Om_1(t)}\BT_1(\Bv,\pi)+\chi_{\Om_2(t)}\BT_2(\Bv,\pi),\quad
	\BT_i(\Bv,\pi)
		=-\pi\BI+2\mu_{i}(|\BD(\Bv)|^2)\BD(\Bv), \quad i=1,2, 
\end{equation*}
and where $|\BD(\Bu)|^2=\sum_{i,j=1}^N(D_{ij}(\Bu))^2$.
The function $\chi_D$ denotes the indicator function of a set $D \subset \ws$, and
the density $\rho$ is defined by $\rho:=\chi_{\Om_1(t)}\rho_1+\chi_{\Om_2(t)}\rho_2$ for the densities $\rho_i>0$ of the $i$-th fluid.
The system is complemented by  the initial fluid velocity  $\Bv_0$, the given initial height function $h_0$ and $\Om_0$
as well as $\Ga_0$ which are given by
\begin{equation*}
	\Om_0=\ws\setminus\Ga_0,\quad
	\Ga_0=\{(x^\pr,x_N) \mid x^\pr\in\tws,\ x_N=h_0(x^\pr)\}.
\end{equation*}
In addition, we denote the unit normal field on $\Ga_0$ by $\Bn_0$. The quantity $\jump{f}=\jump{f}(x,t)$ is the jump of the quantity $f$,
which is defined on $\Om(t)$, across the interface $\Ga(t)$ as 
\begin{equation*}
\jump{f}(x,t) =\lim_{\ep\to0+}\{f(x+\ep\Bn_\Ga,t)-f(x-\ep\Bn_\Ga,t)\}\quad\text{for $x\in\Ga(t)$}.
\end{equation*}
The problem then is to find functions  $\Bv$, $\pi$ and $\Gamma$ solving equation \eqref{NS}.  

Well-posedness results for the above system \eqref{NS} in the case of  {\it Newtonian fluids}
and in the special case of one-phase flows with or without surface tension were first obtained by  Solonnikov  
\cite{Sol87}, \cite{Sol99}, \cite{Sol04}, Shibata and Shimizu \cite{S-S}, \cite{SS07}.
The case of an ocean of infinite extend and which is bounded below by a solid surface and bounded above
by a free surface was treated by  Beale \cite{Bea84}, Allain \cite{All87}, Tani \cite{Tan96},
Tani and Tanaka \cite{TT95}, Bae \cite{Bae11}, and Denk, Geissert, Hieber, Saal and Sawada 
\cite{DGH} and G\"otz \cite{Dario}. 

Besides the articles cited already above, the two-phase problem for Newtonian fluids was studied   by Denisova in \cite{Den91} and \cite{Den94},
and by Tanaka in \cite{Tan95} using {\it Lagrangian coordinates}.
Indeed, Denisova  proved local wellposedness in the Newtonian case in $W^{r,r/2}_2$ for $r \in (5/2,3)$
for the case that one of the domains is bounded
and Tanaka obtained wellposedness (including thermo-capillary convection)
in $W^{r,r/2}_2$ for $r \in (7/2,4)$. 
 
Pr\"uss and Simonett were using in  \cite{P-S3}, \cite{P-S1} and \cite{P-S2} a different approach
by transforming problem \eqref{NS} to a problem on a fixed domain via the {\it Hanzawa transform}, which then was 
followed then by an optimal regularity approach for the linearized equations.
Like this they proved wellposedness of the above problem in the case of Newtonian fluids. 

For an approach to the linearized problem based on Lagrangian coordinates , also in the setting of Newtonian fluids,
we refer to the work of Shibata and Shimizu \cite{S-S1}.

Problems of the above kind for {\it non-Newtonian fluids} were treated by Abels in \cite{Abels}
in the context of measure-valued varifold solutions.  His result covers in particular 
the situation where  $\mu_i(s)=\nu_i s^{(d-2)/2}$ for $i=1,2$ and $d\in(1,\infty)$.
Note, however, that his approach  does not give the uniqueness of a solution. For further results we refer also to the work of Abels, Dienig and 
Terasawa in \cite{ADT}. G\"otz investigated in \cite{Dario} the spin-coating process for generalized Newtonian fluids
and showed local wellposedness of this problem for the setting of one-phase flows.

Bothe and Pr\"uss gave in \cite{B-P} a local wellposedness result for Non-Newtonian fluids on fixed domains for viscosity functions 
$\mu \in C^{1}(0,\infty)$ satisfying $\mu(s)>0$ and $\mu(s)+2\mu'(s)>0$ for $s\geq 0$.
Note that our assumptions on the viscosity function $\mu$ are different from those treated in \cite{B-P}.
For further results on various  classes of non-Newtonian fluids on fixed domains we refer e.g. to
the articles \cite{D-R}, \cite{FMS03}, \cite{MNR} and \cite{Pr-R}.   

In our main result we show that system \eqref{NS} admits a unique, strong solution on $(0,T)$
for arbitrary $T>0$ provided the viscosity functions $\mu_i$ fulfill 
\eqref{viscosity} and the initial data are sufficiently small in their natural norms.  More precisely, we have the following result. 

\begin{theo}\label{thm:main}
Let $N+2<p<\infty$ and $J=(0,T)$ for some $T>0$.
Suppose that $\rho_1>0,\rho_2>0,\ga_a\geq0,\si>0$ and that 
$$
\mu_i \in C^3([0,\infty)) \quad  \text{and}\quad \mu_i(0)>0, \quad i=1,2.
$$
Then there exists $\ep_0>0$ such that for 
\begin{equation*}
(\Bv_0,h_0) \in W_p^{2-2/p}(\Om_0)^N\times W_p^{3-2/p}(\tws)
\end{equation*}
satisfying  the compatibility conditions
\begin{align*}
\jump{\mu(|\BD(\Bv_0)|^2)\BD(\Bv_0)\Bn_0-\{\Bn_0\cdot\mu(|\BD(\Bv_0)|^2)\BD(\Bv_0)\Bn_0\}\Bn_0} &=0\quad\text{on $\Ga_0$}, \\
\di\Bv_0=0\quad\text{in $\Om_0$},\quad
\jump{\Bv_0} &=0\quad \text{on $\Ga_0$},
\end{align*}
as well as the smallness condition
\begin{equation*}
		\|\Bv_0\|_{W_p^{2-2/p}(\Om_0)^N}+\|h_0\|_{W_p^{3-2/p}(\tws)}<\ep_0,
\end{equation*}
the  system  \eqref{NS} admits a unique solution $(\Bv,\pi,h)$ within the class 
\begin{align*}
\Bv &\in  H_p^1(J,L_p(\Om(t)) \cap L_p(J,H_p^2(\Om(t)))^N, \\
\pi &\in  L_p(J,\dot{H}_p^1(\Om(t))), \\
h &\in   W_p^{2-1/(2p)}(J,L_p(\tws))\cap H_p^1(J,W_p^{2-1/p}(\tws)) \\
&\quad \cap W_p^{1/2-1/(2p)}(J,H_p^2(\tws))\cap L_p(J,W_p^{3-1/p}(\tws)).
\end{align*}
\end{theo}

\begin{remarks}
a) 
Some remarks on notation are in order at this point. Setting   
\begin{equation*}
\dws=\ws\setminus\bdry,\quad\bdry=\{(x^\pr,x_N)\mid x^\pr\in\tws,\ x_N=0\},
\end{equation*}
by $\Bv\in H_p^1(J, L_p(\Om(t)))\cap L_p(J,H_p^2(\Om(t)))^N$ we mean that 
\begin{equation*}
\Te^*\Bv=\Bv\circ\Te\in H_p^1(J,L_p(\dws))\cap L_p(J,H_p^2(\dws))^N,
\end{equation*}
where $\Te$ and $\Te^*$ are defined in Section 2 by (\ref{trans}) and (\ref{pull}), respectively.
The regularity statement for $\pi$ is understood in the same way. \\
b) The assumption $p>N+2$ implies that 
\begin{equation*}
h\in BUC(J,BUC^2(\tws)),\quad\pa_t h\in BUC(J,BUC^1(\tws)),
\end{equation*}
which means that the condition on the free interface can be understood in the classical sense. \\
c) Typical examples of viscosity functions $\mu$ satisfying our conditions are given by  
\begin{equation*}
\begin{aligned}
\mu(s) &=\nu (1+s^\frac{d-2}{2})&&\text{with $d=2,4,6$, or $d\geq 8$}, \\
\mu(s) &=\nu(1+s)^{\frac{d-2}{2}}&&\text{with $1\leq d<\infty$}
\end{aligned}
\end{equation*}
for $\nu>0$. For more information and details we refer e.g. to the work of \cite{FMS03}, \cite{D-R}, \cite{MNR} and \cite{Pr-R}. 
Obviously, if $d=2$, all  viscosity functions corresponds to the Newtonian situation.
\end{remarks}

Let us remark at this point that our  proof of Theorem \ref{thm:main} is inspired by the work by Pr\"uss and Simonett in \cite{P-S1} and \cite{P-S2}.
Our strategy may be described as follows: in Section \ref{sec:red} we transform the system \eqref{NS} to a problem on a fixed domain.
Maximal regularity properties of the associated linearized problem due to Pr\"uss and Simoonett \cite{P-S2} are described  in Section \ref{sec:linear}. Of special importance 
will be the function space  $\wt{\BBF}_3(a)$ which will be introduced  and investigated in Section 4.
Finally, in  Section \ref{sec:nonl}, we treat the nonlinear problem and give a proof of our  main theorem.  

In the following, the letter $C$ denote a  generic constant which value may change from line to line.

\section{Reduction to a fixed domain}\label{sec:red}

We start this section by calculating the divergence of the stress tensor,
i.e. by calculating explicitly 
$$
\di\{\mu_d(|\BD(\Bu)|^2)\BD(\Bu)\} \quad \mbox{for} \quad d=1,2.
$$ 
Let us remark first that given a vector $\Bu$ of length $m$ for $m\geq2$, we denote by $u_i$ its $i$-th component and
by $\Bu^\pr$ its tangential component, i.e. 
$\Bu=(u_1,\ldots,u_m)^T$ and $\Bu^\pr=(u_1,\dots,u_{m-1})^T$.
We then obtain 
\begin{equation*}
(\di \{\mu_d(|\BD(\Bu)|^2)\BD(\Bu)\})_i 
=\frac{1}{2}\sum_{j,k,l=1}^N\{2{\mu}'_d(|\BD(\Bu)|^2)D_{ij}(\Bu)D_{kl}(\Bu)
+\mu_d(|\BD(\Bu)|^2)\de_{ik}\de_{jl}\}(\pa_j\pa_k u_l+\pa_j\pa_l u_k).
\end{equation*}
For vectors $\Bu,\,\Bv$ we set  $\BA_d(\Bu)\Bv:=(A_{d,1}(\Bu)\Bv,\dots,A_{d,N}(\Bu)\Bv)^T$ where  
\begin{align*}
	A_{d,i}(\Bu)\Bv
		&:=-\sum_{j,k,l=1}^N A_{d,i}^{j,k,l}(\BD(\Bu))(\pa_j\pa_k v_l+\pa_j\pa_l v_k) \quad \mbox{and}, \\
	A_{d,i}^{j,k,l}(\BD(\Bu))
		&:=\frac{1}{2}\Big(2\dot{\mu}_d(|\BD(\Bu)|^2)D_{ij}(\Bu)D_{kl}(\Bu)+\mu_d(|\BD(\Bu)|^2)\de_{ik}\de_{jl}\Big) \notag
\end{align*}
for $d=1,2$ and $i=1,\dots,N$. We then have 
\begin{equation*}
	\BA_d(\Bu)\Bu
		=-\di\{\mu_d(|D(\Bu)|^2)D(\Bu)\}\quad \mbox{and}\quad 
	\BA_d(0)\Bu
		=-\mu_d(0)(\De\Bu+\nabla\di\Bu).
\end{equation*}
In addition, we set
\begin{equation*}
	\BA(\Bu)\Bv
		:=\chi_{\Om_1(t)}\BA_1(\Bu)\Bv+\chi_{\Om_2(t)}\BA_2(\Bu)\Bv\quad \mbox{and}\quad
	\wt{\pi}
		:=\pi+\rho\ga_a x_N.
\end{equation*}
The system  (\ref{NS}) may thus be rewritten as  
\begin{equation}\label{eq:red}
	\left\{\begin{aligned}
		\rho(\pa_t\Bv+\Bv\cdot\nabla\Bv)-\mu(0)\De\Bv+\nabla\wt{\pi}
			&=-(\BA(\Bv)-\BA(0))\Bv && \text{in $\Om(t)$,} \\
		\di\Bv
			&=0 && \text{in $\Om(t)$,} \\
		-\jump{\wt{\BT}\Bn_\Ga}
			&=\si H_\Ga\Bn_\Ga+\jump{\rho}\ga_a x_N && \text{on $\Ga(t)$,} \\
		\jump{\Bv}
			&= 0 && \text{on $\Ga(t)$,} \\
		\BV_\Ga
			&=\Bv\cdot\Bn_\Ga && \text{on $\Ga(t)$,} \\
		\Bv|_{t=0}
			&=\Bv_0 && \text{in $\Om_0$,} \\
		\Ga|_{t=0}
			&=\Ga_0,
	\end{aligned}\right.
\end{equation}
where $\wt{\BT}=\chi_{\Om_1(t)}\BT_1(\Bv,\wt{\pi})+\chi_{\Om_2(t)}\BT_2(\Bv,\wt{\pi})$ and $\mu(0)=\chi_{\Om_1(t)}\mu_1(0)+\chi_{\Om_2(t)}\mu_2(0)$.

Next, we transform the problem (\ref{eq:red}) to a problem on the fixed domain $\dws$.
To this end, we define a transformation $\Te$ on $J\times\dws$ for $J=(0,T)$ with $T>0$ as 
\begin{align}\label{trans}
	\Te: J\times\dws\ni(\tau,\xi^\pr,\xi_N)
		\mapsto(t,x^\pr,x_N)\in\bigcup_{s\in J}\{s\}\times\Om(s),  
	\text{ with $t=\tau,\ x^\pr=\xi^\pr,\ x_N=\xi_N+h(\xi^\pr,\tau)$} 
\end{align}
for some scalar-valued function $h$. Note that $\det{\CJ\Te}=1$, where $\CJ\Te$ denotes the Jacobian matrix of $\Te$. 
We now define 
\begin{equation}\label{pull}
	\Bu(\Btau,\xi)
		:=\Te^* \Bv(t,x)
		:=\Bv(\Te(\tau,\xi)), \quad
	\te(\Btau,\xi)
		:=\Te^* \pi(t,x),
\end{equation}
as well as 
\begin{equation}\label{push}
	\Te_* f(\tau,\xi):=f(\Te^{-1}(x,t))\quad\text{for $f:\dws\to\ws$},
\end{equation}
where $\Te^{-1}$ given by  $\Te^{-1}(x,t)=(t,x^\pr,x_N-h(t,x^\pr))$.
This change of coordinates implies the relations
\begin{align}\label{deriv}
	\pa_t
		=\pa_\tau-(\pa_\tau h)D_N,\quad
	\pa_j
		=D_j-(D_j h)D_N,\quad
	\pa_j\pa_k
		=D_jD_k-\CF_{jk}(h), \quad \mbox{where}\\
	\CF_{jk}(h)
		:=(D_jD_k h)D_N+(D_j h)D_ND_k+(D_kh)D_jD_N-(D_jh)(D_kh)D_N^2 \notag
\end{align}
for $j,k=1,\dots,N$, $\pa_\tau=\pa/\pa \tau$ and $D_j=\pa/\pa\xi_j$ since $D_N h=0$.

Setting $\De^\pr=\sum_{j=1}^{N-1}\pa_j^2$ and $\nabla^\pr=(\pa_1,\dots,\pa_{N-1})^T$, we first obtain similarly as in  (\ref{deriv})
\begin{equation}\label{E}
	\BD_x(\Bv)
		=\BE(\Bu,h)
		:=\BD_\xi(\Bu)-\CE(\Bu,h),\quad
	\CE(\Bu,h)
		:=(D_N\Bu)\begin{bmatrix}\nabla_{\xi}^\pr h \\ 0 \end{bmatrix}^T
			+\begin{bmatrix}\nabla_{\xi}^\pr h \\ 0\end{bmatrix}(D_N\Bu)^T.
\end{equation}
Secondly, following \cite[Section 2]{P-S1} we see that  
\begin{align*}
	H_\Ga
		&=\sum_{j=1}^{N-1}D_j\Big(\frac{\nabla_{\xi^\pr}^\pr h(t,\xi^\pr)}{\sqrt{1+|\nabla_{\xi^\pr}^\pr h(t,\xi^\pr)|^2}}\Big)
		=\De_{\xi^\pr}^\pr h-\CH(h), \quad \mbox{where} \\
	\CH(h)
		&:=\frac{|\nabla_{\xi^\pr}^\pr h|^2\De_{\xi^\pr}^\pr h}{(1+\sqrt{1+|\nabla_{\xi^\pr}^\pr h|^2})\sqrt{1+|\nabla_{\xi^\pr}^\pr h|^2}}
			+\sum_{j,k=1}^{N-1}\frac{(D_j h)(D_k h)(D_j D_k h)}{(1+|\nabla_{\xi^\pr}^\pr h|^2)^{3/2}}.
\end{align*}
Hence, system (\ref{eq:red}) is reduced to the following problem on $\dws$
\begin{equation}\label{eq:flat}
	\left\{\begin{aligned}
	\rho\pa_\tau\Bu-\mu(0)\De\Bu+\nabla\te
		&=\BF(\Bu,\te,h) && \text{in $\dws$,} \\
	\di\Bu
		&=F_d(\Bu,h) && \text{in $\dws$,} \\
	-\jump{\mu(0)(D_N u_j+D_j u_N)}
		&=G_j(\Bu,\jump{\te},h) && \text{on $\bdry$,} \\
	\jump{\te}-2\jump{\mu(0)D_Nu_N}-(\jump{\rho}\ga_a+\si\De^\pr)h
		&=G_N(\Bu,h) && \text{on $\bdry$,} \\
	\jump{\Bu}
		&=0 && \text{on $\bdry$,} \\
	\pa_\tau h-u_N
		&=G_h(\Bu^\pr,h) && \text{on $\bdry$,}  \\
	\Bu|_{t=0}
		&=\Bu_0 && \text{on $\dws$,} \\
	h|_{t=0}
		&=h_0 && \text{on $\tws$,}
	\end{aligned}\right.
\end{equation}
where $j=1,\dots,N-1$ and $\BF=(F_1,\dots,F_N)^T$. The terms on the right hand side of  (\ref{eq:flat}) are given by
\begin{eqnarray*}
	F_i(\Bu,\te,h)
		 &:=&\rho\{(\pa_\tau h)D_N u_i-(\Bu\cdot\nabla)u_i+(\Bu^\pr\cdot\nabla^\pr h)D_Nu_i\} 
			-\mu(0)\sum_{j=1}^N\CF_{jj}(h)u_i+(D_i h)D_N\te+\CA_i(\Bu,h)  \\
	G_j(\Bu,\jump{\te},h)
		 &:=&\si\CH(h)D_jh-\{(\jump{\rho}\ga_a+\si\De^\pr)h\}D_jh
			+\jump{\te}D_jh+\CB_j(\Bu,h)  \\
	G_N(\Bu,h)
		 &:=& -\si\CH(h)+\CB_N(\Bu,h), \\
	F_d(\Bu,h)
		 &:=& (D_N\Bu^\pr)\cdot\nabla^\pr h  = D_N(\Bu^\pr\cdot\nabla^\pr h),\\
	G_h(\Bu,h)
		 &:=& -\Bu^\pr\cdot\nabla^\pr h.
\end{eqnarray*}
Here $\CA_i(\Bu,h)$, $\CB_j(\Bu,h)$ and $\CB_N(\Bu,h)$ are given by 
\begin{align*}
	\CA_i(\Bu,h)
		 :=&\sum_{j,k,\ell=1}^N
				\Big(A_i^{j,k,\ell}(\BE(\Bu,h))-A_i^{j,k,\ell}(0)\Big)
				(D_jD_ku_\ell+D_jD_\ell u_k) \notag \\
			& -\sum_{j,k,\ell=1}^N
				\Big(A_i^{j,k,\ell}(\BE(\Bu,h))-A_i^{j,k,\ell}(0)\Big)
				(\CF_{jk}(h)u_\ell+\CF_{j\ell}(h)u_k), \quad i=1,\ldots,N, \notag  \displaybreak[0] \\
	\CB_j(\Bu,h)
		:=&-\jump{\mu(|\BE(\Bu,h)|^2)D_Nu_N}D_j h \notag 
			+\jump{\{\mu(|\BE(\Bu,h)|^2)-\mu(0)\}(D_Nu_j+D_j u_N)} \notag \\
			&  -\sum_{k=1}^{N-1}\jump{\mu(|\BE(\Bu,h)|^2)(D_j u_k+D_k u_j)}D_kh \notag \\
			&  +\sum_{k=1}^{N-1}\jump{\mu(|\BE(\Bu,h)|^2)(D_Nu_j D_k h+D_N u_k D_j h)}D_k h, \quad j= 1,\ldots, N-1, \notag \displaybreak[0] \\
	\CB_N(\Bu,h)
		:=& 2\jump{\{\mu(|\BE(\Bu,h)|^2)-\mu(0)\}D_Nu_N} \notag 
			+\jump{\mu(|\BE(\Bu,h)|^2)D_Nu_N}|\nabla^\pr h|^2 \notag \\
			&  -\sum_{k=1}^{N-1}\jump{\mu(|\BE(\Bu,h)|^2)(D_Nu_k+D_ku_N)}D_kh \notag
\end{align*}
where  
$$
A_i^{j,k,l}(\BE(\Bu,h)):=\chi_{\lhs}A_{i,1}^{j,k,l}(\BE(\Bu,h))+\chi_{\uhs}A_{i,2}^{j,k,l}(\BE(\Bu,h)).
$$
In particular, note that
\begin{eqnarray*}
	\mu(|\BE(\Bu,h)|^2)
		&=&\chi_{\lhs}\mu_1(|\BE(\Bu,h)|^2)+\chi_{\uhs}\mu_2(|\BE(\Bu,h)|^2),\\
	\rho
		&=&\chi_{\lhs}\rho_1+\chi_{\uhs}\rho_2, \\
	\mu(0)
		&=&\chi_{\lhs}\mu_1(0)+\chi_{\uhs}\mu_2(0).
\end{eqnarray*}
Finally, in order to simplify our notation we set
\begin{align*}
	\BG(\Bu,\jump{\te},h)
		&:=(G_1(\Bu,\jump{\te},h),\dots,G_{N-1}(\Bu,\jump{\te},h),G_N(\Bu,h))^T  \\
	{\CA}(\Bu,h)
		&:=(\CA_1(\Bu,h),\dots,\CA_N(\Bu,h))^T, \\
	{\CB}(\Bu,h)
		&:=(\CB_1(\Bu,h),\dots,\CB_N(\Bu,h))^T.
\end{align*}

\section{The linearized problem}\label{sec:linear}
The above set of equations \eqref{eq:flat} leads to the following associated linear problem 
\begin{equation}\label{eq:lin}
	\left\{\begin{aligned}
		\rho\pa_t\Bu-\nu\De\Bu+\nabla\te
			&=\Bf && \text{in $\dws$,} \\
		\di\Bu
			&=f_d && \text{in $\dws$,} \\
		-\jump{\nu(D_Nu_j+D_ju_N)}	
			&=g_j && \text{on $\bdry$}, \\
		\jump{\te}-2\jump{\mu D_Nu_N}-(\jump{\rho}\ga_a+\si\De^\pr)h
			&=g_N && \text{on $\bdry$,} \\
		\jump{\Bu}
			&=0 && \text{on $\bdry$,} \\
		\pa_t h-u_N
			&=g_h && \text{on $\bdry$,} \\
		\Bu|_{t=0}
			&=\Bu_0 && \text{in $\dws$}, \\
		h|_{t=0}
			&=h_0 && \text{on $\tws$},
	\end{aligned}\right.
\end{equation}
where $j=1,\dots,N-1$ and  $\Bg=(g_1,\dots,g_N)^T$. Here,
\begin{equation*}
	\rho
		=\rho_1\chi_{\lhs}+\rho_2\chi_{\uhs},\quad
	\nu
		=\nu_1\chi_{\lhs}+\nu_2\chi_{\uhs}
\end{equation*}
with $\rho_i>0$ and $\nu_i>0$ for $i=1,2$.

The optimal regularity property of the solution of the above problem \eqref{eq:lin} will be of central importance in the following.  
To this end, let us  recall  first the definition of some  function spaces. Indeed, let $m\in\BBN$, $\Om \subset \R^N$ be an open set and $X$ be a Banach space. Then,  
for $1 < p < \infty$  and $s \in \R$, the Bessel potential space of order s is denoted by $H_p^s(\Om,X)$. Moreover, given $1\leq p<\infty$ and 
$s\in(0,\infty)\setminus\BBN$, $W_p^s(\Om,X)$ denotes the Sobolev-Slobodeckii space equipped with  the norm
$$
	\|f\|_{W_p^s(\Om,X)}
		=\|f\|_{W_p^{[s]}(\Om,X)}
			+\sum_{|\al|=[s]}(\int_\Om\int_\Om\frac{\|\pa^\al f(x)-\pa^\al f(y)\|_X^p}{|x-y|^{m+(s-[s])p}}\,dxdy)^{1/p},
$$
where $[s]$ is the largest non-negative integer smaller than $s$ and $\pa^\al f(x)=\pa^{|\al|}f(x)/\pa x_1^{\al _1}\dots\pa x_m^{\al_m}$. By $BUC(\Om,X)$ we denote 
the Banach space of uniformly continuous and bounded functions on $\Om$ and $BUC^k(\Om,X)$ denotes the set of functions in $C^k(\Om,X)$ such that all derivatives up to order $k$ 
are belonging to $BUC(\Om,X)$ for $k\in\BBN$. For $1\leq p<\infty$, the homogeneous Sobolev space $\dot{H}_p^1(\Om)$ of order $1$ is  defined as
$\dot{H}_p^1(\Om) =\{f\in L_{1,\text{loc}}(\Om)\mid \|\nabla f\|_{L_p(\Om)}<\infty\}$. Finally, $\dot{H}_p^{-1}(\ws)$ denotes the dual space of $\dot{H}_{p^\pr}^1(\ws)$ for $1/p+1/p^\pr=1$.

The following result due to Pr\"uss and Simonett \cite{P-S1} characterizes the set of data on the right-hand sides of \eqref{eq:lin} for one  
obtains a solution of \eqref{eq:lin} in the maximal regularity space.

\begin{prop} \cite[Theorem 5.1]{P-S1}, \cite[Theorem 3.1]{P-S2}. \label{lemm:lin}
Let $1<p<\infty$, $p\neq3/2,3$, $a>0$ and  $J=(0,a)$.
Suppose that
\begin{equation*}
		\rho_i>0,\ \nu_i>0,\ \ga_a \geq0\ \text{and $\si>0$}, \quad i =1,2.
\end{equation*}
Then, equation \eqref{eq:lin} admits a unique solution $(\Bu,\te,h)$ 
\begin{align*}
		\Bu
			&\in (H_p^1(J,L_p(\dws))\cap L_p(J,H_p^2(\dws)))^N, \\
		\te
			&\in L_p(J,\dot{H}_p^1(\dws)), \\
		\jump{\te}
			&\in W_p^{1/2-1/(2p)}(J,L_p(\tws))\cap L_p(J,W_p^{1-1/p}(\tws)), \\
		h
			&\in W_p^{2-1/(2p)}(J,L_p(\tws))  \cap H_p^1(J,W_p^{2-1/(2p)}(\tws)) \cap L_p(J,W_p^{3-1/p}(\tws))
	\end{align*}
if  and only if the data $(\Bf,f_d,\Bg,g_h,\Bu_0,h_0)$ satisfy the following regularity and compatibility conditions:
\begin{eqnarray*}
\Bf  &\in& L_p(J,L_p(\dws))^N,  \\
f_d  &\in& H_p^1(J,\dot{H}_p^{-1}(\ws))\cap L_p(J,H_p^1(\dws)),  \\
\Bg  &\in& (W_p^{1/2-1/(2p)}(J,L_p(\tws))\cap L_p(J,W_p^{1-1/p}(\tws)))^N,  \\
g_h  &\in& W_p^{1-1/(2p)}(J,L_p(\tws))\cap L_p(J,W_p^{2-1/p}(\tws)),  \\
\Bu_0  &\in& W_p^{2-2/p}(\dws)^N, \\ 
h_0 &\in& W_p^{3-2/p}(\tws),  \\
\di\Bu_0  &=& f_d(0)\ \text{in $\dws$},\\
 \jump{\Bu_0}&=& 0\quad \text{on $\tws$ if $p>3/2$},  \\
g_j(0) &=& -\jump{\nu(D_Nu_{0j}+D_ju_{0N})} =g_j(0) \quad \text{on $\tws$ if $p>3$}
\end{eqnarray*}
for all $j=1,\ldots,N-1$. Moreover, the solution map $[(f,f_d,g,g_h,\Bu_0,h_0)\mapsto(\Bu,\te,h)]$ is continuous between the corresponding spaces.
\end{prop}

\section{Properties of function spaces involved}
In order to derive estimates for the nonlinear mappings occuring on the right-hand sides of \eqref{eq:flat}
we study first embedding properties of the functions spaces involved. 
For $a>0$ let $J=(0,a)$ and set
\begin{align*}
	\BBE_1(a)
		& =\{\Bu\in (H_p^1(J,L_p(\dws))\cap L_p(J,H_p^2(\dws)))^N\mid \jump{\Bu}=0\}, \\
	\BBE_2(a)
		& =L_p(J,\dot{H}_p^1(\dws)), \notag \\
	\BBE_3(a)
		& =W_p^{1/2-1/(2p)}(J,L_p(\tws))\cap L_p(J,W_p^{1-1/p}(\tws)),\notag\\
	\BBE_4(a)
		& =W_p^{2-1/(2p)}(J,L_p(\tws))\cap H_p^1(J,W_p^{2-1/p}(\tws))\notag\\
			&\qquad\cap W_p^{1/2-1/(2p)}(J,H_p^2(\tws))\cap L_p(J,W_p^{3-1/p}(\tws))\notag
\end{align*} 
as well as 
\begin{align*}
	\BBF_1(a)
		& =L_p(J,L_p(\dws))^N, \\
	\BBF_2(a)
		& =H_p^1(J,\dot{H}_p^{-1}(\ws))\cap L_p(J,H_p^1(\dws)),\notag\\
	\BBF_3(a)
		& =(W_p^{1/2-1/(2p)}(J,L_p(\tws))\cap L_p(J,W_p^{1-1/p}(\tws)))^N,\notag\\
	\BBF_4(a)
		& =W_p^{1-1/(2p)}(J,L_p(\tws))\cap L_p(J,W_p^{2-1/p}(\tws)). \notag
\end{align*}

We then have the following result due to Pr\"uss and Simonett \cite{P-S1}. 

\begin{lemm} \cite[Lemma 6.1]{P-S1}. \label{lemm:em}
Let $N+2<p<\infty$, $a>0$ and $J=(0,a)$. Then the following properties hold true.
\begin{enumerate}
[{\rm a)}] \item\label{lemm:em_0}
$\BBE_3(a)$ and $\BBF_4(a)$ are multiplication algebras.
\item\label{lemm:em_1}
$\BBE_1(a)\hookrightarrow (BUC(J,BUC^1(\dws))\cap BUC(J,BUC(\ws)))^N$ and $\BBE_1(a)\hookrightarrow W_p^{1/2}(J,H_p^1(\dws))^N$.
\item\label{lemm:em_6}
$\BBE_3(a)\hookrightarrow BUC(J,BUC(\tws))$.
\item\label{lemm:em_2}
$\BBE_4(a)\hookrightarrow BUC^1(J,BUC^1(\tws))\cap BUC(J,BUC^2(\tws))$.
\item 
$W_p^{2-1/(2p)}(J,L_p(\tws))\cap H_p^1(J,W_p^{2-1/p}(\tws)) \cap L_p(J,W_p^{3-1/p}(\tws))\hookrightarrow\BBE_4(a)$ 
\end{enumerate}
\end{lemm}

The crucial point of our proof is the investigatation  of the viscosity functions $\mu$.
To this end, given $a>0$, we introduce  the function space $\wt{\BBF}_3(a)$ as  
\begin{equation*}
	\wt{\BBF}_3(a)
		:=\{g\in BUC(J,BUC(\tws)): \|g\|_{\wt{\BBF}_3(a)}=\|g\|_{BUC(J,BUC(\tws))}+|g|_{\BBF_3(a)}<\infty\},
\end{equation*}
where $|g|_{\BBF_3(a)}=|g|_{\BBF_3(a),1} + |g|_{\BBF_3(a),2}$ with
\begin{align*}
	|g|_{\BBF_3(a),1}
		&:=\Big(\int_J\int_J
				\frac{\|g(t)-g(s)\|_{L_p(\tws)}^p}{|t-s|^{\frac{1}{2}+\frac{p}{2}}}\,dtds
				\Big)^{1/p}\enskip \mbox{and} \\
	|g|_{\BBF_3(a),2}
		&:=\Big(\int_J\int_{\tws}\int_{\tws}
				\frac{|g(t,x^\pr)-g(t,y^\pr)|^p}{|x^\pr-y^\pr|^{N-2+p}}\,dx^\pr dy^\pr dt
				\Big)^{1/p}.
\end{align*}

We then obtain the following result. 

\begin{lemm}\label{F-tilde}
Let $N+2<p<\infty$, $a>0$, and $J=(0,a)$. Then the following properties hold true.	
\begin{enumerate}[{\rm a)}]
\item $\BBF_3(a)$ and $\wt{\BBF}_3(a)$ are multiplication algebras. In addition, 
\begin{equation*}
			\BBF_3(a)\hookrightarrow BUC(J,BUC(\tws))^N \quad \mbox{ and } \quad 
				\wt{\BBF}_3(a)\hookrightarrow BUC(J,BUC(\tws)).	
\end{equation*}
\item If $\ph\in BUC^1(\BBR)$ and $g\in\wt\BBF_3(a)$, then
\begin{equation*}
			\|\ph(g)\|_{\wt{\BBF}_3(a)}
				\leq \|\ph\|_{BUC(\BBR)}+\|\dot{\ph}\|_{BUC(\BBR)}|g|_{\BBF_3(a)}.
\end{equation*}
\item
There exists a positive constant $C$ such that 
\begin{equation*}
\|fg\|_{\BBF_3(a)} \leq C\|f\|_{\BBF_3(a)}\|g\|_{\wt{\BBF}_3(a)}\quad
\text{for $f\in\BBF_3(a)$ and $g\in\wt{\BBF}_3(a)$}. 
\end{equation*}
\end{enumerate}
\end{lemm}

\begin{proof} 
a) The properties for $\BBF_3(a)$ is essentially given  in Lemma \ref{lemm:em} (\ref{lemm:em_0}) and (\ref{lemm:em_6}). The embedding $\wt{\BBF}_3(a)\hookrightarrow BUC(J,BUC(\tws))$ 
follows from the definition of $\wt{\BBF}_3(a)$. We thus only show  that $\wt{\BBF}_3(a)$ is a multiplication algebra. For $f,g\in\wt{\BBF}_3(a)$ it follows that
\begin{align*}
				\|fg\|_{BUC(J,BUC(\tws))}
					\leq \|f\|_{BUC(J,BUC(\tws))}\|g\|_{BUC(J,BUC(\tws))} 
					\leq \|f\|_{\wt{\BBF}_3(a)}\|g\|_{\wt{\BBF}_3(a)}.
\end{align*}
Considering $|\cdot|_{\BBF_3(a),1}$ we see that 
\begin{align*}
|fg|_{\BBF_3(a),1}
&\leq \|f\|_{BUC(J,BUC(\tws))} (\int_J\int_J\frac{\|g(t)-g(s)\|_{L_p(\tws)}^p}{|t-s|^{\frac{1}{2}+\frac{p}{2}}}\,dtds)^{1/p} \notag \\
						& \quad +\|g\|_{BUC(J,BUC(\tws))}
							\Big(\int_J\int_J\frac{\|f(t)-f(s)\|_{L_p(\tws)}^p}{|t-s|^{\frac{1}{2}+\frac{p}{2}}}\,dtds\Big)^{1/p} \notag \\
&\leq\|f\|_{\wt{\BBF}_3(a)}|g|_{\BBF_3(a),1}+\|g\|_{\wt{\BBF}_3(a)}|f|_{\BBF_3(a),1}.
\end{align*}
Similarly, $|fg|_{\BBF_3(a),2}\leq \|f\|_{\wt{\BBF}_3(a)}|g|_{\BBF_3(a),2}+\|g\|_{\wt{\BBF}_3(a)}|f|_{\BBF_3(a),2}$. 
This yields $\|fg\|_{\wt{\BBF}_3(a)}\leq C\|f\|_{\wt{\BBF}_3(a)}\|g\|_{\wt{\BBF}_3(a)}$,
which implies that $\wt{\BBF}_3(a)$ is a multiplication algebra. \\
b) By the mean value theorem
\begin{align*}
					|\ph(g)|_{\BBF_3(a)}
						= & \, \Big(\int_J\int_J
								\frac{\|\ph(g(t))-\ph(g(s))\|_{L_p(\tws)}^p}{|t-s|^{\frac{1}{2}+\frac{p}{2}}}\,dtds
								\Big)^{1/p} \\
							&+\Big(\int_J\int_{\tws}\int_{\tws}
								\frac{|\ph(g(t,x^\pr))-\ph(g(t,y^\pr))|^p}{|x^\pr-y^\pr|^{N-2+p}}\,dx^\pr dy^\pr dt
								\Big)^{1/p} \displaybreak[0] \\
						\leq & \, \|\dot{\ph}\|_{BUC(\BR)}\Big\{\Big(\int_J\int_J
								\frac{\|g(t)-g(s)\|_{L_p(\tws)}^p}{|t-s|^{\frac{1}{2}+\frac{p}{2}}}\,dtds
								\Big)^{1/p} \\
							&+\Big(\int_J\int_{\tws}\int_{\tws}
								\frac{|g(t,x^\pr)-g(t,y^\pr)|^p}{|x^\pr-y^\pr|^{N-2+p}}\,dx^\pr dy^\pr dt
								\Big)^{1/p}\Big\} \displaybreak[0] \\
						\leq & \, \|\dot{\ph}\|_{BUC(\BBR)}|g|_{\BBF_3(a)},
				\end{align*}
which yields the required inequality. \\
c) Obviously, 
\begin{align*}
					\|fg\|_{L_p(J,L_p(\tws))}
						\leq \|f\|_{L_p(J,L_p(\tws))}\|g\|_{BUC(J,BUC(\tws))} 
						\leq \|f\|_{\BBF_3(a)}\|g\|_{\wt{\BBF}_3(a)}.
				\end{align*}
On the other hand, we see that by Lemma \ref{F-tilde} (1) and by calculations similar to a) there exists a constant $C>0$
such that for $i=1,2$ 
\begin{equation*}
					|fg|_{\BBF_3(a),i}
						\leq \|f\|_{BUC(J,BUC(\tws))}|g|_{\BBF_3(a),i}+\|g\|_{BUC(J,BUC(\tws))}|f|_{\BBF_3(a),i}
						\leq C\|f\|_{\BBF_3(a)}\|g\|_{\wt{\BBF}_3(a)},
\end{equation*}
which combined with the above inequality completes the proof.
\end{proof}

We next recall basic properties of functions which are Fr\'echet differentiable. 
Let $X$ and $Y$ be  Banach spaces and $U\subset X$ be open.
We then denote the Fr\'{e}chet derivative of a differentiable mapping
$\Phi:U\to Y$ by $D\Phi:U\to\CL(X,Y)$ and its evaluation for $u\in U$ and $v\in X$ by $[D\Phi(u)]v\in Y$.
Moreover, a mapping $\Phi:U\to Y$ is called continuously Fr\'echet differentiable 
if and only if $\Phi$ is Fr\'echet differentiable on $U$ and its  Fr\'echet derivative $D\Phi$ is continuous on $U$.
The set of such continuously Fr\'echet differentiable mappings from $U$ to $Y$ is denoted by $C^1(U,Y)$.

In the sequel, we will make use of the chain and product rule for Fr\'echet differentiable functions.
In fact, in addition let $Z$ be a further Banach space and 
suppose that the mappings $f:U\to Y$ and  $g:Y\to Z$ are continuously Fr\'echet differentiable.
Then the composition $F=g\circ f: U\to Z$ is also continuously 
Fr\'echet differentiable and its evaluation at $x\in U$ and $\bar{x}\in X$ is given by
\begin{equation*}
		[DF(x)]\bar{x}
			=[Dg(f(x))][Df(x)]\bar{x}.
	\end{equation*} 
For the product rule, suppose that there exists a constant $M>0$ such that for every $y\in Y$ and $z\in Z$
\begin{equation*}
		\|yz\|_Y\leq M\|y\|_Y\|z\|_Z,
\end{equation*}
and also that $f:U\to Y$ and $g:U\to Z$ are continuously Fr\'echet differentiable. Set $F(x)=f(x)g(x)$ for $x\in U$.
Then $F:U\to Y$ is also continuously Fr\'echet differentiable and its evaluation at $x\in U$ and $\bar{x}\in X$ is given by
\begin{equation*}\label{prodform}
		[DF(x)]\bar{x}
			=g(x)[Df(x)]\bar{x}+f(x)[Dg(x)]\bar{x}.
\end{equation*}

Now, we define the solution space  $\BBE(a)$ and the data space $\BBF(a)$ for $a>0$ by 
\begin{align*}
	\BBE(a)
		&:=\{(\Bu,\te,\pi,h)\in\BBE_1(a)\times\BBE_2(a)\times\BBE_3(a)\times\BBE_4(a)
			\mid \jump{\te}=\pi\}, \\
	\BBF(a)
		&:=\BBF_1(a)\times\BBF_2(a)\times\BBF_3(a)\times\BBF_4(a).
\end{align*}
The spaces $\BBE(a)$ and $\BBF(a)$ are endowed with their natural norms, i.e. 
\begin{align*}
	\|(\Bu,\te,\pi,h)\|_{\BBE(a)}
		:&=\|\Bu\|_{\BBE_1(a)}
			+\|\te\|_{\BBE_2(a)}+\|\pi\|_{\BBE_3(a)}+\|h\|_{\BBE_4(a)}, \\
	\|(\Bf,f_d,\Bg,g_h)\|_{\BBF(a)}
		:&=\|\Bf\|_{\BBF_1(a)}
			+\|f_d\|_{\BBF_2(a)}+\|\Bg\|_{\BBF_3(a)}+\|g_h\|_{\BBF_4(a)}.
\end{align*}
Finally, we consider for $(\Bu,\te,\pi,h)\in\BBE(a)$ the nonlinear mapping $\BN$ which is defined as 
\begin{equation}\label{defN}
	\BN(\Bu,\te,\pi,h)
		:=(\BF(\Bu,\te,h),F_d(\Bu,h),\BG(\Bu,\pi,h),G_h(\Bu,h)),
\end{equation}
where the terms on the right hand side  are defined as  in Section \ref{sec:red}. 
For functions $\Bu=(u_1,\dots,u_N)^T$ defined on $\dws$ we set 
\begin{align}\label{pm}
	\Bu^1
		:&=(u_1^1,\dots,u_N^1),\quad
	u_j^1	
		:=\chi_{\lhs}u_j,\\
	\Bu^2
		:&=(u_1^2,\dots,u_N^2),\quad
	u_j^2
		:=\chi_{\uhs}u_j.\notag
\end{align}
Recalling  the definition of  $\BE(\Bu,h)$ and $\CE(\Bu,h)$ in \eqref{E}, the following lemma and its corollary shows that various functions 
occuring in the definition of $N$ in \eqref{defN} are  Fr\'echet dfferentiable.

\vspace{.2cm}\noindent
\begin{lemm}\label{Frechet1}
Let $N+2<p<\infty$, $a>0$ and $J=(0,a)$. Then the following assertions  hold true.
\begin{enumerate}[{\rm a)}]
\item For $\phi\in BUC^1(\BBR)$, the mapping
\begin{equation*}
				\ph:BUC(J,BUC(\dws))\to BUC(J,BUC(\dws))
\end{equation*}
is continuously Fr\'echet differentiable.
\item For $\psi \in BUC^3(\BBR)$, the mapping	
\begin{equation*}
				\psi:\wt{\BBF}_3(a)\to\wt{\BBF}_3(a)
			\end{equation*}
is continuously Fr\'echet differentiable.
\item Let $\phi \in BUC^1(\BBR)$ and for $u$ defined as in \eqref{pm} set 
\begin{equation*}
\Phi^d(\Bu,h) :=\ph(|\BE(\Bu^d,h)|^2)\quad\text{for $d=1,2$}.
\end{equation*}
Then $\Phi^d:\BBE_1(a)\times\BBE_4(a)\to BUC(J,BUC(\dws))$ is continuously Fr\'echet differentiable.
\item Let $\psi \in BUC^3(\BBR)$ and for $u$ defined as in \eqref{pm} set 
\begin{equation*}
\Psi^d(\Bu,h):=\psi(|\ga_0\BE(\Bu^d,h)|^2)\quad\text{for $d=1,2$},
\end{equation*}
where $\ga_0$ denotes the trace to $\bdry$. Then $\Psi^d:\BBE_1(a)\times\BBE_4(a)\to\wt{\BBF}_3(a)$ is continuously Fr\'echet differentiable.
\end{enumerate}
\end{lemm}

 \begin{proof}
a) We show first that the mapping $\ph$ is Fr\'echet differentiable. To this end, let $f,\bar{f}\in Z:= BUC(J,BUC(\dws))$. Then
\begin{align*}
				\ph(f+\bar{f})-\ph(f)-\dot{\ph}(f)\bar{f}
					=&\int_0^1(\dot{\ph}(f+\te\bar{f})-\dot{\ph}(f))\,d\te\bar{f},
			\end{align*}
which implies 
\begin{align*}
\|\ph(f+\bar{f})-\ph(f)-\dot{\ph}(f)\bar{f}\|_{Z}/\|\bar{f}\|_{Z}  \leq \int_0^1\|\dot{\ph}(f+\te\bar{f})-\dot{\ph}(f)\|_{Z}\,d\te.
\end{align*}
Since $\dot{\ph}\in BUC(\BBR)$, the term on the right hand side above tends to $0$ as $\|\bar{f}\|_{BUC(J,BUC(\dws))}\to 0$. Thus 
$[D\ph(f)]\bar{f}=\dot{\ph}(f)\bar{f}$. Next, we show the continuity of the Fr\'echet derivative at $f_0\in Z$.
For $h\in Z$ we have  
\begin{eqnarray*}
\|D\ph(f_0+h)-D\ph(f_0)\|_{\CL(Z)}  &=& \sup_{\|f\|_{Z}=1} \|[D\ph(f_0+h)]f-[D\ph(f_0)]f\|_{Z} \\ 
				& =& \sup_{\|\bar{f}\|_{Z}=1} \|\dot{\ph}(f_0+h)f-\dot{\ph}(f_0)f\|_{Z} \\
				&\leq& \|\dot{\ph}(f_0+h)-\dot{\ph}(f_0)\|_{Z},
\end{eqnarray*}
which tends to $0$ as  $\|h\|_{Z}\to 0$ since $\dot{\ph}\in BUC(\BBR)$.\\
b) For  $f,\bar{f}\in\wt{\BBF}_3(a)$ we obtain 
\begin{align*}
				\psi(f+\bar{f})-\psi(f)-\dot{\psi}(f)\bar{f}
					=\int_0^1(1-\te)\ddot{\psi}(f+\te\bar{f})\bar{f}\bar{f}\,d\te.
\end{align*}
By Lemma \ref{F-tilde} a) and b)
\begin{align*}
				\|\ddot{\psi}(f+\te\bar{f})\bar{f}\|_{\wt{\BBF}_3(a)}
					\leq& C\{\|\ddot{\psi}\|_{BUC(\BBR)}+\|\dddot{\psi}\|_{BUC(\BBR)}|f+\te\bar{f}|_{\BBF_3(a)}\}\|\bar{f}\|_{\wt{\BBF}_3(a)}^2 \\
					\leq& C(1+\|f\|_{\wt\BBF_3(a)}+\te\|\bar{f}\|_{\wt\BBF_3(a)})\|\bar{f}\|_{\wt{\BBF}_3(a)}^2,
			\end{align*} 
which implies that $[D\psi(f)]\bar{f}=\dot{\psi}(f)\bar{f}$.
Next, in order to show  the continuity of the Fr\'echet derivative at $f_0\in\wt{\BBF}_3(a)$, let 
$h\in\wt{\BBF}_3(a)$. Then 
\begin{align*}
				\|D\psi(f_0+h)-D\psi(f_0)\|_{\CL(\wt{\BBF}_3(a))}
					=&\sup_{\|f\|_{\wt{\BBF}_3(a)}=1}\|[D\psi(f_0+h)]f-[D\psi(f_0)]f\|_{\wt{\BBF}_3(a)}\\					
					\leq& C\|\dot{\psi}(f_0+h)-\dot{\psi}(f_0)\|_{\wt\BBF_3(a)}
			\end{align*}
by Lemma \ref{F-tilde} a).
Since $\dot{\psi}\in BUC^2(\BBR)$, Lemma \ref{F-tilde} b) implies that 
$\dot{\psi}(f_0+h)$ and  $\dot{\psi}(f_0)$ are  in $\wt{\BBF}_3(a)$.
Taylor's formula and Lemma \ref{F-tilde} a) yield 
\begin{equation*}
				\|\dot{\psi}(f_0+h)-\dot{\psi}(f_0)\|_{\wt{\BBF}_3(a)}
					\leq C\int_0^1\|\ddot{\psi}(f_0+\te h)\|_{\wt{\BBF}_3(a)}\,d\te\|h\|_{\wt{\BBF}_3(a)}.
\end{equation*}
The latter terms  tends to $0$ as $\|h\|_{\wt{\BBF}_3(a)}\to0$, which completes the proof. \\
c) By Lemma \ref{lemm:em} b) and d), the mappings
\begin{equation}\label{0709_1}
(\Bu,h) \mapsto \BE^d(\Bu,h):\BBE_1(a)\times\BBE_4(a)\to Z^{N\times N}\enskip\mbox{as well as}\enskip
x \mapsto |x|^2:Z^{N\times N}\to Z 
\end{equation}
are continuously Fr\'echet differentiable. The chain rule thus yields that 
\begin{equation*}
			\BBE_1(a)\times\BBE_4(a)\to Z: 	(\Bu,h) \mapsto |\BE(\Bu^d,h)|^2
 \end{equation*}
is continuously Fr\'echet differentiable, too.
Applying  assertion a) and the chain rule again implies that $\Phi^d:\BBE_1(a)\times\BBE_4(a)\to Z$ is 
continuously Fr\'echet differentiable. \\
d) Note that for $(\Bu,h)\in \BBE_1(a)\times\BBE_4(a)$ and $i,j=1,\ldots,N$,  we obtain 
\begin{equation}\label{0707_1}
			\|(\ga_0\pa_i u_{j}^d,\pa_i h,\pa_i\pa_j h)\|_{\wt{\BBF}_3(a)}
					+\|(\ga_0\pa_i\Bu^d,\nabla h,\pa_i\nabla h)\|_{\BBF_3(a)}
				\leq C(a,p)\|z\|_{\BBE(a)}
		\end{equation}
In fact, Lemma \ref{F-tilde} b) and \cite[Theorem 4.5]{M-S} for $s=1/2$, $m=1$, and $\mu=1$ yield 
\begin{align*}
\|\ga_0\pa_i u_j^d\|_{\wt{\BBF}_3(a)}+\|\ga_0\pa_i\Bu^d\|_{\BBF_3(a)} 
				\leq&\|\ga_0\pa_i u_j^d\|_{BUC(J,BUC(\tws))}+\|\ga_0\pa_i\Bu^d\|_{\BBF_3(a)} \notag \\
				\leq& C(\|\pa_i u_j^d\|_{BUC(J,BUC(\uhs))}+\|\pa_i\Bu^d\|_{W_p^{1/2}(J,L_p(\uhs))\cap L_p(J,H_p^1(\uhs))}) \notag \\
				\leq& C\|\Bu\|_{\BBE_1(a)},
\end{align*}
which implies the required properties of $\Bu$. Concerning $h$, the desired properties  follow from the definition of $\BBE_4(a)$ and Lemmas \ref{lemm:em} and \ref{lemm:em_2}).
By (\ref{0707_1}) and Lemma \ref{F-tilde} a) the mappings 
\begin{equation*}
			(\Bu,h)\mapsto \ga_0\BE(\Bu^d,h):\BBE_1(a)\times\BBE_4(a)\to\wt{\BBF}_3(a)^{N\times N} \mbox{ and }
			\ x\mapsto|x|^2: \wt{\BBF}_3(a)^{N\times N}\to \wt{\BBF}_3(a)
		\end{equation*}
are continuously Fr\'echet differentiable. By the chain rule
\begin{equation*}
			(\Bu,h)\mapsto |\ga_0\BE(\Bu^d,h)|^2:\BBE_1(a)\times\BBE_4(a)\to \wt{\BBF}_3(a)
\end{equation*}
is continuously Fr\'echet differentiable. Finally, combining this with assertion b), the chain rule yields the assertion. 
\end{proof}

For $\Phi \in BUC^1(\BBR)$, $u$ as in  \eqref{pm} and $i,j,k,l,m,q,r=1,\ldots,N$ and $d=1,2$ we now introduce the functions   
\begin{align*}
				\Phi_{ijklmqr}^d(\Bu,h)
					&:=\ph(|\BE(\Bu^d,h)|^2)E_{ij}(\Bu^d,h)E_{kl}(\Bu^d,h)\pa_m\pa_q u_r^d, \\
				\La_{ijklpqr}^d(\Bu,h)
					&:=\ph(|\BE(\Bu^d,h)|^2)E_{ij}(\Bu^d,h)E_{kl}(\Bu^d,h)\CF_{mq}(h)u_r^d, \\
				\Phi_{ijk}^d(\Bu,h)
					&:=(\ph(|\BE(\Bu^d,h)|^2)-\ph(0))\pa_i\pa_j u_k^d ,\\
				\La_{ijk}^d(\Bu,h)
					&:=(\ph(|\BE(\Bu^d,h)|^2)-\ph(0))\CF_{ij}(h)u_k^d,
\end{align*}
as well as for $\psi\in BUC^3(\BBR)$, $u$ as in \eqref{pm}, $i,j,k=1,\dots,N$ and $d=1,2$ the functions 
\begin{align*}
				\Psi_{i}^d(\Bu,h)
					&:=\{\psi(|\ga_0\BE(\Bu^d,h)|^2)-\psi(0)\}\ga_0\pa_i\Bu^d, \\
				\Te_{ij}^d(\Bu,h)
					&:=\psi(|\ga_0\BE(\Bu^d,h)|^2)(\ga_0\pa_i\Bu^d)\pa_j h, \\
				\Xi_{ijk}^d(\Bu,h)
					&:=\psi(|\ga_0\BE(\Bu^d,h)|^2)(\ga_0\pa_i\Bu^d)\pa_j h\pa_k h.
			\end{align*}

\vspace{.2cm}\noindent
\begin{coro}\label{coro}
Let $N+2<p<\infty$, $a>0$ and $J=(0,a)$. Then the following assertions hold true. 
\begin{enumerate}[{\rm a)}]
\item Assume that $i,j,k,l,m,q,r=1,\ldots,N$ and $d=1,2$. For $\phi \in BUC^1(\BBR)$ and  $u$ as in  \eqref{pm} the mappings
\begin{align*}
				\Phi_{ijklmqr}^d,\La_{ijklmqr}^d,\Phi_{ijk}^d,\La_{ijk}^d:
					\BBE_1(a)\times\BBE_4(a)\to L_p(J,L_p(\dws))
\end{align*}
are continuously Fr\'echet differentiable. Moreover, their values  and their Fr\'echet derivatives at $(\Bu,h)=(0,0)$ vanish.
\item Let $i,j,k=1,\dots,N$ and $d=1,2$. For $\psi\in BUC^3(\BBR)$ and $u$ as in \eqref{pm} the functions 
\begin{align*}
				\Psi_{i}^d,\Te_{ij}^d,\Xi_{ijk}^d: \BBE_1(a)\times\BBE_4(a)\to \BBF_3(a)
			\end{align*}	
are continuously Fr\'echt differentiable.
Moreover, their values and Fr\'echet derivatives at $(\Bu,h)=(0,0)$ vanish.
\end{enumerate}
\end{coro}

\begin{proof}
a) We only prove the assertion for $\Phi_{ijklmqr}^d$; the remaining assertions can be proved in an analogously.  
By (\ref{0709_1}), Lemma \ref{Frechet1} a) and the product rule, the function 
\begin{equation*}
				(\Bu,h)\mapsto \ph(|\BE(\Bu^d,h)|^2)E_{ij}(\Bu^d,h)E_{kl}(\Bu^d,h):
					\BBE_1(a)\times\BBE_4(a)\to BUC(J,BUC(\dws))
			\end{equation*}
is continuously Fr\'echet differentiable. Moreover, 
\begin{equation*}
				\Bu\mapsto \pa_m\pa_q u_r^d:
					\BBE_1(a)\to L_p(J,L_p(\dws))
			\end{equation*}
is continuously Fr\'echet differentiable, which combined with the above assertion and the product rule applied to the situation  $X=\BBE_1(a)\times\BBE_4(a)$,
$Y=L_p(J,L_p(\dws))$ and $Z=BUC(J,BUC(\dws))$ implies that  $\Phi_{ijklmqr}^d:\BBE_1(a)\times\BBE_1(a)\to L_p(J,L_p(\dws))$ is continuously Fr\'echet differentiable.
In addition, it is clear that $\Phi_{ijklmqr}^d(0,0)=0$ and $D\Phi_{ijklmqr}^d(0,0)=0$. \\
b) We only prove the assertion for  $\Psi_i^d$. Since 
\begin{equation}\label{0709_3}
				\Bu\mapsto\ga_0\pa_i\Bu^d:
					\BBE_1(a)\to\BBF_3(a)\ \text{or $\wt{\BBF}_3(a)$}
\end{equation}
is continuously Fr\'echet differentiable, it follows  from \eqref{0707_1} that 
\begin{equation*}
				\Bu\mapsto\psi(0)\ga_0\pa_i\Bu^d
					:\BBE_1(a)\to\BBF_3(a)
			\end{equation*}
is continuously Fr\'echet differentiable, too.
On the other hand by Lemma \ref{Frechet1} b), \eqref{0709_3} and the product rule applied to 
$X=\BBE_1(a)\times\BBE_4(a)$, $Y=\BBF_3(a)$ and $Z=\wt{\BBF}_3(a)$, the mapping  
\begin{equation*}
(\Bu,h)\mapsto\psi(|\ga_0\BE(\Bu^d,h)|^2)\ga_0\pa_i\Bu^d:
\BBE_1(a)\times\BBE_4(a)\to\BBF_3(a)
\end{equation*}
is continuously Fr\'echet differentiable. Finally, $\Psi_i^d(0,0)=0$ and thanks to product rule 
\begin{align*}
				D\Psi_i^d(\Bu,h) 
					=\ga_0\pa_i\Bu^d [D(\psi(|\ga_0\BE(\Bu^d,h)|^2)-\psi(0))]
						+(\psi(|\ga_0\BE(\Bu^d,h)|^2)-\psi(0))[D(\ga_0\pa_i\Bu^d)],
			\end{align*}
which implies that $D\Psi_i^d(0,0)=0$. The proof is  complete .
\end{proof}

\section{The nonlinear problem}\label{sec:nonl}

Let us recall from \eqref{defN} that for  $(\Bu,\te,\pi,h)\in\BBE(a)$ the nonlinear mapping $\BN$  was  defined as 
\begin{equation*}
	\BN(\Bu,\te,\pi,h) =(\BF(\Bu,\te,h),F_d(\Bu,h),\BG(\Bu,\pi,h),G_h(\Bu,h)).
\end{equation*}
We start this section by examining properties of the nonlinear mapping $\BN$.  

\begin{lemm}\label{N}
Let $N+2<p<\infty$, $a>0$ and $r>0$.
Suppose that $\mu_d(s)\in C^3([0,\infty))$ for $d=1,2$ and in addition
that $\rho_1>0$, $\rho_2>0$, $\ga_a\geq 0$ and $\si>0$ are constants. Then
\begin{equation*}
		\BN\in C^1(B_{\BBE(a)}(r),\BBF(a)),\ \BN(0)=0\ \text{and $D\BN(0)=0$.}
\end{equation*}
\end{lemm}

\begin{proof}
We treat here in detail only the terms $\BCA(\Bu,h)$ and $\BCB(\Bu,h)$
which are defined in Section \ref{sec:red}. The remaining terms may be treated as in
\cite[Proposition 6.2]{P-S1} and \cite[Proposition 4.1]{P-S2}.\\ 
{\it The term  $\BCA(\Bu,h)$}: \\
Let $(\Bu,\te,\pi,h)\in B_{\BBE(a)}(r)$ and recall that $\CA_i$ is given for $i=1,\dots,N$ by 
\begin{align*}
			\CA_i(\Bu,h)
				=&\sum_{j,k,l=1}^N(A_{i}^{j,k,l}(\BE(\Bu,h))-A_{i}^{j,k,l}(0))(\pa_j\pa_k u_l+\pa_j\pa_l u_k) \\
					&+\sum_{j,k,l=1}^N(A_{i}^{j,k,l}(\BE(\Bu,h))-A_{i}^{j,k,l}(0))(\CF_{jk}(h)u_l+\CF_{jl}(h) u_k), \notag
		\end{align*}
where $\CF_{jk}(h)$ are defined as in (\ref{deriv}) and
$A_{i}^{j,k,l}$ as in Section \ref{sec:red} by 
\begin{align*}
			A_{i}^{j,k,l}(\BE(\Bu,h))
				&=\chi_{\lhs} A_{1,i}^{j,k,l}(\BE(\Bu,h))+\chi_{\uhs} A_{2,i}^{j,k,l}(\BE(\Bu,h)), \\
			A_{d,i}^{j,k,l}(\BE(\Bu,h))
				&=\frac{1}{2}\Big(2\dot{\mu}_d(|\BE(\Bu,h)|^2)E_{ij}(\Bu,h)E_{kl}(\Bu,h)+\mu_d(|\BE(\Bu,h)|^2)\de_{ik}\de_{jl}\Big)
		\end{align*}
for $d=1,2$. Observe that $A_{i}^{j,k,l}(\BE(\Bu,h))$ may be represented as 
\begin{equation*}
A_i^{j,k,l}(\BE(\Bu,h))
=\frac{1}{2}\sum_{d=1}^2
\Big(2\dot{\mu}_d(|\BE(\Bu^d,h)|^2)E_{ij}(\Bu^d,h)E_{kl}(\Bu^d,h)+\mu_d(|\BE(\Bu^d,h)|^2)\de_{ik}\de_{jl}\Big),
\end{equation*}
and thus
\begin{align*}
\CA_i(\Bu,h)
=&\sum_{d=1}^2\sum_{j,k,l=1}^N\dot{\mu}_d(|\BE(\Bu^d,h)|^2)E_{ij}(\Bu^d,h)E_{kl}(\Bu^d,h)(\pa_j\pa_k u_l^d+\pa_j\pa_l u_k^d) \\
&+\frac{1}{2}\sum_{d=1}^2\sum_{j,k,l=1}^N(\mu_d(|\BE(\Bu^d,h)|^2)-\mu_d(0))\de_{ik}\de_{jl}(\pa_j\pa_k u_l^d+\pa_j\pa_l u_k^d) \\
&+\sum_{d=1}^2\sum_{j,k,l=1}^N\dot{\mu}_d(|\BE(\Bu^d,h)|^2)E_{ij}(\Bu^d,h)E_{kl}(\Bu^d,h)(\CF_{jk}(h)u_l^d+\CF_{jl}(h) u_k^d) \\
&+\frac{1}{2}\sum_{d=1}^2\sum_{j,k,l=1}^N(\mu_d(|\BE(\Bu^d,h)|^2)-\mu_d(0))\de_{ik}\de_{jl}(\CF_{jk}(h)u_l^d+\CF_{jl}(h) u_k^d)
\end{align*}
for $i=1,\dots,N$. By Corollary \ref{coro} a) 
\begin{equation*}
\BCA\in C^1(B_{\BBE(a)}(r),\BBF_1(a)),\ \BCA(0,0)=0\ \text{and $D\BCA(0,0)=0$}.
\end{equation*}
{\it The term  $\BCB(\Bu,h)$:}\\
Let $(\Bu,\te,\pi,h)\in B_{\BBE(a)}(r)$. By Lemma \ref{lemm:em} b) and d),
each term appearing in $\BCB(\Bu,h)$ is continuous with respect to the space variable.
In particular, this implies 
\begin{equation*}
			\ga_0\dot{\mu}_d(|\BE(\Bu^d,h)|^2))=\dot{\mu}_d(|\ga_0\BE(\Bu^d,h)|^2)),
				\quad\ga_0\{(\pa_N u_N^d)(\pa_j h)\}=(\ga_0\pa_N u_N^d)(\ga_0\pa_j h).
		\end{equation*}
Thus, $\BCB(\Bu,h)$ may be rewritten as
\begin{align*}
			\CB_{j}(\Bu,h)
				=& -\sum_{d=1}^2(-1)^{d}\mu_d(|\ga_0\BE(\Bu^d,h)|^2)(\ga_0\pa_Nu_N^d)\pa_jh \\
					& +\sum_{d=1}^2(-1)^{d}(\mu_d(|\ga_0\BE(\Bu^d,h)|^2)-\mu_d(0))(\ga_0\pa_N u_j^d+\ga_0\pa_j u_N^d) \\
					& -\sum_{d=1}^2\sum_{k=1}^{N-1}(-1)^{d}\mu_d(|\ga_0\BE(\Bu^d,h)|^2)(\ga_0\pa_j u_k^d+\ga_0\pa_k u_j^d)\pa_k h \\
					& +\sum_{d=1}^2\sum_{k=1}^{N-1}(-1)^{d}\mu_d(|\ga_0\BE(\Bu^d,h)|^2)
		(\pa_kh \ga_0\pa_N u_j^d+\pa_j h\ga_0\pa_Nu_k^d)\pa_k h,\\
			\CB_N(\Bu,h)
				=& 2\sum_{d=1}^2(-1)^{d}(\mu_d(|\ga_0\BE(\Bu^d,h)|^2)-\mu_d(0))\ga_0\pa_N u_N^d \\
					& +\sum_{d=1}^2\sum_{k=1}^{N-1}(-1)^{d}\mu_d(|\ga_0\BE(\Bu^d,h)|^2)(\pa_k h)^2\ga_0\pa_N u_N^d \\
					& -\sum_{d=1}^2\sum_{k=1}^{N-1}(-1)^{d}\mu_d(|\ga_0\BE(\Bu^d,h)|^2)(\ga_0\pa_Nu_k^d+\ga_0\pa_k u_N^d)\pa_k h.
		\end{align*}
These representation  combined with Corollary \ref{coro} b) yields 
\begin{equation*}
			\BCB\in C^1(B_{\BBE(a)}(r),\BBF_3(a)),\ \BCB(0,0)=0\ \text{and $D\BCB(0,0)=0$},
		\end{equation*}
which implies the assertion. 
\end{proof}

Finally, we return to the nonlinear problem (\ref{eq:flat}). We define the space of initial data $\BBI$ by 
\begin{equation*}
	\BBI:= W_p^{2-2/p}(\dws)^N\times W_p^{3-2/p}(\tws),
\end{equation*}
and define also for $z\in \BBE(a)$ and $(\Bu_0,h_0)\in\BBI$ the mapping $\Phi$ by   
\begin{equation*}
	\Phi(z):=L^{-1}(\BN(z),\Bu_0,h_0).
\end{equation*}
Here $L$ is defined by the left-hand side of the linear problem (\ref{eq:lin}) with $\nu=\mu(0)$.
Observe that the invertiblility of $L$ is guaranteed by Proposition 
\ref{lemm:lin} since  $\mu_i(0)>0$ for $i=1,2$ by assumption and
$\BN(z)\in\BBF(a)$ for  $\Bz\in\BBE(a)$ by Lemma \ref{N}.
The following result shows that the problem \eqref{eq:flat} on the fixed domain admits a unique strong solution provided the 
data $u_0$ and $h_0$ are sufficiently small in their corresponding norms.

\begin{prop}\label{thm:fixed}
Let $N+2<p<\infty$ and $a>0$. Suppose that $\mu_i \in C^3([0,\infty))$ for $i=1,2$ and that 
\begin{equation*}
		\rho_i>0,\ \mu_i(0)>0,\ \ga_a \geq0\ \text{and $\si>0$}.
	\end{equation*}
Then there exist positive constants $\ep_0$ and $\de_0$ (depending on $a$ and $p$),
such that system  \eqref{eq:flat} admits a unique solution $(\Bu,\te,h)$ in $B_{\BBE(a)}(\de_0)$ provided that the initial data $(\Bu_0,h_0)\in\BBI$ satisfy the compatibility conditions
\begin{align}\label{comp_2}
		\jump{\mu(|\BE(\Bu_0,h_0)|^2)\BE(\Bu_0,h_0)\Bn_0-\{\Bn_0\cdot\mu(|\BE(\Bu_0,h_0)|^2)\BE(\Bu_0,h_0)\Bn_0\}\Bn_0}
			=0 \quad\text{on $\bdry$}, \\
		\jump{\te_0}
			=\jump{\Bn_0\cdot\mu(|\BE(\Bu_0,h_0)|^2)\BE(\Bu_0,h_0)\Bn_0}
				+(\jump{\rho}\ga+\si\De^\pr)h_0-\si\CH(h_0)\quad\text{on $\bdry$}, \notag \\
		\di\Bu_0
			=F_d(\Bu_0,h_0)\quad\text{in $\dws$}, \quad
		\jump{\Bu_0}
			=0\quad\text{on $\bdry$} \notag
 	\end{align}
as well as  the smallness condition $\|(\Bu_0,h_0)\|_{\BBI}<\ep_0$.
\end{prop}

\begin{rema}
The compatibility conditions (\ref{comp_2}) are equivalent to 
\begin{align*}
		-\jump{\mu(0)(\pa_N u_j+\pa_j u_N)}=G_j(\Bu_0,\jump{\te_0},h_0)\quad\text{on $\bdry$},\\
		\jump{\te_0}-2\jump{\mu(0)\pa_N u_{0N}}-(\jump{\rho}\ga_a+\si\De^\pr)h_0=G_N(\Bu_0,h_0)\quad\text{on $\bdry$}, \\
		\di u_0= F_d(\Bu_0,h_0)\quad\text{in $\dws$},\quad\jump{\Bu_0}=0\quad\text{on $\bdry$}
	\end{align*}
for  $j=1,\dots,N-1$ and $\BG=(G_1,\dots,G_N)^T$ defined as in Section \ref{sec:red}.
\end{rema}

\begin{proof}
Observe  that $D\BN$ is continuous and $D\BN(0)=0$ by Lemma \ref{N}.
Thus, we may choose  $\de_0>0$ small enough such that
\begin{equation*}
		\sup_{z\in B_{\BBE(a)}(2\de_0)}\|D\BN(z)\|_{\CL(B_{\BBE(a)}(r),\BBF(a))}
			\leq\frac{1}{2\|L^{-1}\|_{\CL(\BBF(a)\times\BBI,\BBE(a))}},
	\end{equation*}
where $r>0$ is a sufficiently large number. For $z\in B_{\BBE(a)}(\de_0)$, the mean value theorem implies 
\begin{align*}
			\|\Phi(z)\|_{\BBE(a)}
				&\leq\|L^{-1}\|_{\CL(\BBF(a)\times\BBI,\BBE(a))}\{\|\BN(z)\|_{\BBF(a)}+\|(\Bu_0,h_0)\|_{\BBI}\} \\
				&=\|L^{-1}\|_{\CL(\BBF(a)\times\BBI,\BBE(a))}\{\|\BN(z)-\BN(0)\|_{\BBF(a)}+\ep_0\} \\
				&\leq \|L^{-1}\|_{\CL(\BBF(a)\times\BBI,\BBE(a))}
					\{(\sup_{\bar{z}\in B_{\BBE(a)}(\de_0)}\|D\BN(\bar{z})\|_{\CL(B_{\BBE(a)}(r),\BBF(a))})\|z\|_{\BBE(a)}+\ep_0\} \\
				&<\frac{\de_0}{2}+\ep_0\|L^{-1}\|_{\CL(\BBF(a)\times\BBI,\BBE(a))}.
		\end{align*}
Choosing $\ep_0$ in such way  that $0<\ep_0<\de_0/(2\|L^{-1}\|_{\CL(\BBF(a)\times\BBI,\BBE(a))})$, we obtain $\|\Phi(z)\|_{\BBE(a)}<\de_0$. Hence, 
$\Phi$ is a mapping from $B_{\BBE(a)}(\de_0)$ into itself. 

Let $z_1,z_2\in B_{\BBE(a)}(\de_0)$. Noting that $\Phi(z_1)-\Phi(z_2)=L^{-1}(\BN(z_1)-\BN(z_2),0,0)$, we obtain by the mean value theorem
\begin{align*}
			\|\Phi(z_1)-\Phi(z_2)\|_{\BBE(a)}
				&\leq\|L^{-1}\|_{\CL(\BBF(a)\times\BBI,\BBE(a))}\|\BN(z_1)-\BN(z_2)\|_{\BBF(a)}\\
				&\leq\|L^{-1}\|_{\CL(\BBF(a)\times\BBI,\BBE(a))}
					(\sup_{\bar{z}\in B_{\BBE(a)}(2\de_0)}\|D\BN(\bar{z})\|_{\CL(B_{\BBE(a)}(r),\BBF(a))}))\|z_1-z_2\|_{\BBE(a)} \\
				&\leq\frac{1}{2}\|z_1-z_2\|_{\BBE(a)},
		\end{align*}
which implies that $\Phi$ is a contraction mapping on $B_{\BBE(a)}(\de_0)$. The contraction principle yields the existence of a  unique solution of 
(\ref{eq:flat}) in $B_{\BBE(a)}(\de_0)$.  
\end{proof}

\vspace{.2cm}\noindent
\text{{\it Proof of Theorem} \ref{thm:main}.} 
Note that the compatibility conditions of Theorem \ref{thm:main} are satisfied if and only if (\ref{comp_2}) is satisfied. The 
mapping $\Te_{h_0}$ given by
\begin{equation*}
		\Te_{h_0}(\xi^\pr,\xi_N)
			:=(\xi^\pr,\xi_N+h_0(\xi^\pr))\quad\text{for $(\xi^\pr,\xi_N)\in\dws$}
\end{equation*}	
defines for $h_0\in W_p^{3-2/p}(\tws)$ a $C^2$-diffeomorphism from $\dws$ onto $\Om(0)$ with inverse $\Te_{h_0}^{-1}(x^\pr,x_N)=(x^\pr,x_N-h_0(x^\pr))$.
Thus there exists a constant $C(h_0)$ such that 
\begin{equation*}
		C(h_0)^{-1}\|\Bv_0\|_{W_p^{2-2/p}(\Om_0)^N}
			\leq \|\Bu_0\|_{W_p^{2-2/p}(\dws)^N}
			\leq C(h_0)\|\Bv_0\|_{W_p^{2-2/p}(\Om_0)^N}.
\end{equation*}
Hence, the smallness condition in Theorem \ref{thm:main} implies the smallness condition in Proposition \ref{thm:fixed}. 
Proposition \ref{thm:fixed} then yields a unique solution $(\Bu,\te,h)\in B_{\BBE(a)}(\de_0)$ of (\ref{eq:flat}).
Finally, setting
\begin{equation*}
		(\Bv,\pi)
			=(\Te_*\Bu,\Te_*\pi)
			=(\Bu\circ\Te^{-1},\pi\circ\Te^{-1}),
\end{equation*}
where $\Te_*$ is defined as in \eqref{push}, we obtain a unique solution $(\Bv,\pi,h)$ of the original problem \eqref{NS} with the regularities stated in 
Theorem \ref{thm:main}. The proof is complete. 

\rightline{$\Box$}

\end{document}